\documentclass{article}
\usepackage{latexsym,eepic,makeidx,fancybox}
\usepackage[dvips]{graphicx} \usepackage[latin1]{inputenc} \usepackage[english]{babel} \usepackage{amstext,amssymb,amsmath,amsthm,youngtab}
\title{Limit points of the iterative scaling procedure}
\author{Erik Aas}
\date{}
\theoremstyle{definition}

\newtheorem{lemma}{Lemma}

\newtheorem{proposition}{Proposition}

\newcommand{\p}{^\prime}

\newcommand{\mcl}{\mathcal}
\newcommand{\e}{\varepsilon}
\newcommand{\R}{\mcl{R}}
\newcommand{\C}{\mcl{C}}
\newcommand{\vphi}{\varphi}
\newcommand{\vtheta}{\vartheta}
\begin{document}
\maketitle
\abstract{The iterative scaling procedure (ISP) is an algorithm which computes a sequence of matrices, starting from some given matrix. The objective is to find a matrix 'proportional' to the given matrix, having given row and column sums. In many cases, for example if the initial matrix is strictly positive, the sequence is convergent. In the general case, it is known that the sequence has at most two limit points. When these are distinct, convergence can be slow. We give an efficient algorithm which finds these limit points, invoking the ISP only on instances for which the procedure is convergent.}

\section*{Introduction}

The {\it iterative scaling procedure} (ISP) is an algorithm which, given an $m\times n$ entrywise nonnegative matrix $A$ and positive numbers $r_1,\dots,r_m$, $c_1,\dots,c_n$ attempts to find a matrix diagonally equivalent to $A$, having row sums $r_i$ and column sums $c_j$. Two matrices $A$ and $A\p$ are {\it diagonally equivalent} if there are strictly positive numbers $x_1,\dots,x_m$, $y_1,\dots,y_n$ such that $a_{ij} = x_ia\p_{ij}y_j$ for all $i,j$.

The ISP has been applied in a variety of contexts, the most interesting of which perhaps being the ranking of webpages \cite{Kni}. A discrete version of the algorithm is used by the Z\"urich City Council to distribute seats in parliamentary elections \cite{Puk}.

We proceed by defining the ISP.
Throughout, $A$ will denote a fixed nonnegative $m\times n$ matrix, and $r_1,\dots,r_m,c_1,\dots,c_n$ fixed positive numbers. We further assume that there is no row or column in $A$ containing only zeros.
All matrices considered will be nonnegative entry by entry and denoted by capital letters.
Matrix entries will be denoted by the corresponding lower case letters.
For example, the $(i,j)$ entry in $A$ is denoted $a_{ij}$.
By the {\it row adjustment} (to $r_1,\dots,r_m$) $\R(A)$ of $A$ we mean the matrix whose $(i,j)$ entry is entries $x_ia_{ij}$, where $x_i = \frac{r_i}{\sum_j a_{ij}}$, and we define the {\it column adjustment} $\C(A)$ (to $c_1,\dots,c_n$) similarly.
The numbers $x_i$ will be referred to as {\it row multipliers}.
Note that $\R(A) = A = \C(A)$ in case $A$ has both the desired row and column sums.

The iterative scaling procedure consists of adjusting rows and columns alternatingly, starting with $A$.
The iterates under the scaling procedure are defined to be $B(k) := \R(C(k-1))$, $C(k) = \C(B(k))$ for $k \geq 1$ and $B(1)= \R(A)$.

It is known that the sequences $B(k)$ and $C(k)$ are convergent \cite{CsiShi}, and that if there is any matrix $D$ with both the desired row sums and column sums, with the property that $a_{ij} = 0 \Rightarrow d_{ij} = 0$, then those two limits are equal (\cite{CsiShi}, see also \cite{Pretzel}).
If there is no such matrix $D$, the limits can clearly not be equal.
In fact, if the limits are not equal, but $\sum_i r_i = \sum_j c_j$, then the support of the limit points are not equal to the support of the initial matrix, that is, some entries in the matrix tend to zero during the ISP.
The objective of this note is to describe an algorithm which finds these entries efficiently (in time less than quadratic in input size - using the ISP itself to find these entries can require arbitrarily many steps, though not on realistic instances).
Proving that the algorithm works is non-trivial, and gives some insight into the structure of the limit points.

We will use the following observation in what follows. If we scale the desired row sums by a common factor $t$, then this gives new ISP sequence $B\p(k), C\p(k)$ closely related to the original sequence; we have $B\p(k) = tB(k)$ and $C\p(k) = C(k)$.

\vspace{0.5cm}

{\bf Acknowledgements} \\
I am very thankful to Fabian Reffel who read an early version of this note, finding a serious error and providing many helpful suggestions. 
I would also like to express my gratitude to Kai-Friedrike Oelbsmann and the Augsburg group headed by Prof. Friedrich Pukelsheim for telling me about the present problem and electoral methods in general.
\section*{An example}

We choose the following initial data, writing the desired row and column sums on the borders of the matrix $A$.

$\begin{array}{c|cccc}
 	& 4 & 4 & 2 & 1 \\ \hline
6 & 1 & 0 & 0 & 0 \\
6 & 1 & 1 & 0 & 0 \\
4 & 1 & 1 & 7 & 2 \\
1 & 1 & 1 & 9 & 6 \\
\end{array}$

After $1000$ iterations we get the following matrices:

$B(500) = 
\left(
\begin{array}{cccc}
6       & 0      & 0     &  0\\
0.171   & 5.83   & 0     &  0\\
0.00907 & 0.309  & 2.61  &  1.08\\
0.00132 & 0.0447 & 0.486 &  0.468\\
\end{array}
\right)$ and

$C(500) = 
\left(
\begin{array}{cccc}
3.88     & 0      & 0     & 0\\
0.111    & 3.77   & 0     & 0\\
0.00587  & 0.2    & 1.69  & 0.697\\
0.000851 & 0.0289 & 0.314 & 0.303\\
\end{array}\right).$

The actual (to three decimal places) limit matrices are 

$B = \left(
\begin{array}{cccc}
6 & 0 & 0 & 0 \\
0 & 6 & 0 & 0 \\
0 & 0 & 2.83 & 1.17 \\
0 & 0 & 0.508 & 0.492\\
\end{array}
\right)$
and 
$C = \left(\begin{array}{cccc}
4 & 0 & 0 & 0 \\
0 & 4 & 0 & 0 \\
0 & 0 & 1.7 & 0.705 \\
0 & 0 & 0.305 & 0.295 \\
\end{array}\right).$

Now we give some further examples which will be useful when reading the next sections, using notation and terminology introduced there.

We have $\Psi(B) = (I_1, J_1)$ where $I_1 = \{1,2\}$, $J_1 = \{1,2\}$.
One can check that we have $\vphi(A,r,c) = \Psi(B)$ in this case, as expected.

In step I of the algorithm, we find the splitting consisting of the blocks $(I_1, J_1)$ and $(I_2, J_2)$ where $I_2 = \{3,4\}$ and $J_2 = \{3,4\}$.

In step II applied to the block $(I_1, J_1)$ we find the subset $I_3 = \{1\} \subseteq I_1$ with the property $r(I_3) = 6 = \frac{12}{8}4 = \frac{r(I_1)}{c(J_1)} c(N(I_3) \cap J_1)$.
So the result of applying step II to $(I_1, J_1)$ is $(I_3, J_3)$ and $(I_4, J_4) := (I_1\backslash I_3, J_1\backslash J_3)$.
Applying step II to any of the blocks $(I_2,J_2)$, $(I_3,J_3)$ or $(I_4,J_4)$ yields nothing new, so the final splitting found is $\{(I_2,J_2), (I_3,J_3), (I_4,J_4)\}$, which coincides with the decomposition of $B$.

\section*{Limit points}

Let (for the remainder of this note) $B$, $C$ denote the limits of the sequences $B(k)$, $C(k)$. 

By a {\it splitting} $S$ we mean a set of pairs $(I,J)$ of sets of row respectively column indices, with the property that the sets $\{I: \exists J : (I,J)\in S\}$ form a partition of some underlying set of rows (which will always be the set of rows of $A$ below) and a similar statement holds for the sets $J$.
We call any pair $I,J$ of rows respectively columns a {\it block}.
An elementary refinement of a splitting consists of replacing a pair $(I,J)$ with $(I_1,J_1)$ and $(I_2,J_2)$ such that $I_1,I_2$ partition $I$ and $J_1,J_2$ partition $J$.
If the splitting $S\p$ is obtained by performing a sequence of elementary refinements on the splitting $S$ then we say that $S\p$ is a {\it refinement} of $S$.

By the {\it decomposition} of a matrix $B$ we will mean the splitting \\ $I_1,\dots,I_r$, $J_1,\dots,J_r$ of the row and column sets of $B$ such that $b_{ij} \neq 0\Rightarrow i\in I_k$ and $j\in J_k$ for some $k$, minimal with respect to refinement.

In this section we will describe the decomposition of $B$.
Since the decomposition of $B$ and of $C$ coincide, we will only mention the decomposition of $B$ in what follows.

Let $S(A) := \{(i,j):a_{ij} \neq 0\}$ denote the {\it support} of $A$.
By definition, $S(A) = S(\R(A))= S(\C(A))$. Clearly $S(B) = S(C) \subseteq S(A), \C(B) = C$, and $\R(C) = B$.

For subsets of rows $I$, define $r(I) := \sum_{i\in I} r_i$ and $c(J)$ similarly for subsets $J$ of columns.

Let $x_i$ and $y_j$ be such that $x_ib_{ij} = c_{ij}$ and $c_{ij}y_j = b_{ij}$.
Hence $x_ib_{ij}y_j = b_{ij}$ and thus $x_iy_j = 1$ whenever $(i,j) \in S(B)$. Let $I_1,\dots,I_r$, $J_1,\dots,J_r$ be the decomposition of $B$.
Then $x_i$ is constant for $i\in I_k$ for each $k$ and $y_j$ is constant in $J_k$ for each $k$, and if $i\in I_k$, $j\in J_k$ we have $x_i = 1/y_j = \frac{r(I_k)}{c(J_k)}$.

First, note that for each $k$, the submatrix $B[I_k, J_k]$ is a matrix with row sums $r_i$ and column sums $\frac{r(I_k)}{c(J_k)}c_j$.
Similarly, $C[I_k,J_k]$ has column sums $c_j$ and row sums $\frac{c(J_k)}{r(I_k)}r_i$.
Whenever $k \neq l$, we have $B[I_k, J_l] = C[I_k, J_l] = 0$. 

We will say a block $(I,J)$ is {\it feasible} if there is some $I \times J$ matrix $M$ with row sums $r_i$, column sums $\frac{r(I)}{c(J)}c_j$ and $S(M) \subseteq S(A)$. A splitting is feasible if all its blocks are feasible. So the decomposition of $B$ is clearly feasible.

The {\it quotient} of a block $(I,J)$ is the number $r(I)/c(J)$.

\section*{The first block}

We now give some definitions and lemmas needed for describing the algorithm.

\begin{lemma}
\label{numbers}
Let $p_1,\dots,p_n, q_1,\dots, q_n$ be positive real numbers. Then $\min_i \frac{p_i}{q_i} \leq \frac{p_1+\dots+p_n}{q_1+\dots+q_n} \leq \max_i \frac{p_i}{q_i}$. If any of the two inequalities is in fact an equality, then all the $p_i/q_i$ are equal.
\end{lemma}
\begin{proof} This follows by induction, the case $n=2$ being easy.\end{proof}
It is important to note that $\min_i p_i / \max_j p_j \leq (p_1+\dots+p_n)/(q_1+\dots+q_n)$ is a considerably weaker statement than Lemma \ref{numbers}.

We will use the following simple generalization of a well-known theorem by Philip Hall (see eg. \cite{match}).
\begin{lemma}
\label{hall}
Suppose $r_1+\dots+r_m = t(c_1+\dots+c_n)$ where $t$ is some positive real number. There is a matrix $B$ with $S(B) \subseteq S(A)$ and row sums $r_i$ and column sums $tc_j$ if and only if there is no subset $I$ of rows such that $r(I) > tc(N(I))$, where $N(I) = \{j\in [n]: \exists i : a_{ij} \neq 0\}$.
\end{lemma}

For initial data $A,r,c$, we define a subset $\vphi(A,r,c)$ of rows. In fact $\vphi(A,r,c)$ will depend only on the support of $A$, and not on the non-zero values themselves.

We define $\vphi(A,r,c)$ as the subset $I$ of rows such that $r(I) / c(N(I))$ is maximal and $\# I$ is maximal among those $I$ maximizing $r(I)/c(N(I))$.
Let us prove that $\vphi(A,r,c)$ is well defined.

\begin{lemma}
\label{unique}
	If $I_1$ and $I_2$ satisfy the definition of $\vphi(A,r,c)$, then $I_1 = I_2$.
\end{lemma}
\begin{proof}
We have $r(I_1) + r(I_2) = r(I_1\cup I_2) + r(I_1\cap I_2)$ and $c(N(I_1)) + c(N(I_2)) = c(N(I_1) \cup N(I_2)) + c(N(I_1) \cap N(I_2)) \geq c(N(I_1\cup I_2)) + c(N(I_1\cap I_2))$. Therefore

$$
\frac{r(I_1) + r(I_2)}{c(N(I_1)) + c(N(I_2))} \leq \frac{r(I_1\cup I_2) + r(I_1\cap I_2)}{c(N(I_1\cup I_2)) + c(N(I_1\cap I_2))}.
$$
By Lemma \ref{numbers} either $I_1 \cup I_2$ or $I_1\cap I_2$ shows that neither $I_1$ nor $I_2$ can satisfy the definition of $\vphi(A,r,c)$.
\end{proof}

Using linear programming, it is easy to compute $\vphi(A,r,c)$.

Though it will follow from proposition \ref{equal}, it is interesting to note that we can prove directly that $(I, N(I))$ is feasible, where $I =\vphi(A,r,c)$.
\begin{lemma} 	Let $I = \vphi(A,r,c)$. Then the block $(I, N(I))$ is feasible. \end{lemma}
\begin{proof} Suppose $(I, N(I))$ is not feasible. By Lemma \ref{hall} there is then some $I\p \subseteq I$ such that $r(I\p) > \frac{r(I)}{c(N(I))}c(N(I\p))$, or $\frac{r(I\p)}{c(N(I\p))} > \frac{r(I)}{c(N(I))}$. But this contradicts the choice of $I$. \end{proof}

Consider the decomposition $\mcl{D}\p$ of $B$ and denote by $\Psi(B)$ the block obtained by merging all blocks with maximal quotient into a single block (which will have this same quotient). We will denote by $\mcl{D}$ the decomposition obtained from $\mcl{D}\p$ after this merge. Of course, the blocks in $\mcl{D}$ are all feasible (since this is true for $\mcl{D}\p$).

\begin{proposition}
\label{equal}
	Let $I = \vphi(A,r,c)$. We have $(I, N(I)) = \Psi(B)$.
\end{proposition}
\begin{proof}
	Let $(I_1, J_1) = \Psi(B)$. We wish to prove that $I_1 = I$ and $J_1 = N(I)$.
	We do this in five steps.
\begin{itemize}
	\item $J_1 = N(I_1)$.

		Suppose this is not the case. Then there are $p \in I_1$, $q \notin J_1$ such that $a_{pq} \neq 0$. Denote by $(I_2, J_2)$ the block in the decomposition of $B$ such that $q \in J_2$.
		
		Let $y_j(k)$ be the column multipliers used when computing\\ $\C(B(k)) = C(k)$ and $x_i(k)$ the row multipliers used when computing $\R(\C(B(k))) = B(k+1)$.
		We know that $x_p(k) \to \frac{r(I_1)}{c(J_1)}$ and $y_q(k) \to \frac{c(J_2)}{r(I_2)}$ as $k\to\infty$.
		Therefore $x_p(k)y_q(k) \to \frac{r(I_1)}{c(J_1)}\frac{c(J_2)}{r(I_2)} > 1$.
		Choose $\eta > 0$ and $K$ such that $x_p(k)y_q(k) > 1+\eta$ for all $k\geq K$.
		Hence $b_{pq}(K+n) > (1+\eta)^n b_{pq}(K)\to\infty$ as $n\to\infty$.
		This contradicts the fact that all entries in the $B(k)$ are bounded by $\max(r_1+\dots+r_m, c_1+\dots+c_n)$.

	\item $r(I) / c(N(I)) = r(I_1) / c(J_1)$ and
	\item $I \subseteq I_1$.

		The proofs of these two statements are similar so we do them simultaneously.
		It follows from the previous step that $r(I)/c(N(I)) \geq r(I_1) / c(J_1)$.

		Note that 
		\begin{equation}
			\label{moo}
			\frac{r(I)}{c(N(I))} = \frac{\sum_{(I\p,J\p)} r(I\cap I\p)}
																	{\sum_{(I\p,J\p)} c(N(I)\cap J\p)}
													\leq
														 \frac{\sum_{(I\p,J\p)} r(I\cap I\p)}
																	{\sum_{(I\p,J\p)} c(N(I\cap I\p)\cap J\p)}
		\end{equation}
		where the sums range over all $(I\p, J\p)\in \mcl{D}$ with $I \cap I\p \neq \emptyset$.

	Since $I_1,J_1$ has maximal quotient in $\mcl{D}$, for any $(I\p,J\p) \in \mcl{D}$, $r(I\cap I\p) / c(N(I\cap I\p) \cap J\p)  \leq r(I_1)/c(J_1) \leq r(I)/c(N(I))$.
	By the formula above and Lemma \ref{numbers}, all the numbers $r(I\cap I\p)/c(N(I\cap I\p) \cap J\p)$ therefore have to be equal, and this common value is $r(I)/c(N(I))$.
	
	Now, suppose $r(I_1)/c(N(I_1)) < r(I) / c(N(I))$.
	Take any term $(I\p$, $J\p)$ in (\ref{moo}). We then have $\frac{r(I\cap I\p)}{c(N(I\cap I\p)\cap J\p)} = \frac{r(I)}{c(N(I))} > \frac{r(I_1)}{c(J_1)} \geq \frac{r(I\p)}{c(J\p)}$. But by Lemma \ref{hall}, this contradicts $(I\p, J\p)$ being feasible.

	So $r(I_1)/c(N(I_1)) = r(I)/c(N(I))$.
	Using a similar argument as in the previous paragraph, if we have any $I\p \neq I$ occurring in (\ref{moo}), then this will contradict $(I\p, J\p)$ being feasible. Thus the only term in (\ref{moo}) is $\frac{r(I\cap I_1)}{c(N(I\cap I_1)\cap J_1)}$, and consequently $I\subseteq I_1$.

	\item $I = I_1$

		This follows directly from the previous two steps and lemma \ref{unique}.
		
	\item $N(I) = J_1$

		We have $J_1 = N(I_1) = N(I)$.
\end{itemize}
\end{proof}

\section*{The algorithm}

The algorithm consists of two steps, I and II. We first describe step I.

The output of step I is the splitting given by the blocks $(I_1,J_1)$, $\dots,$ $(I_r,J_r)$ where $(I_1, J_1) = \vphi(A,r,c)$, $(I_2, J_2) = \vphi(A[I\p, J\p], r_{|I\p}, c_{|J\p})$ with $I\p = [m]-I$, $J\p = [n]-J$ and so on. Here we think of $r$ and $c$ as functions, and $r_{|I\p}$ et.c. denotes restriction.

It follows from Lemma \ref{equal} that the decomposition of $B$ is a refinement of the splitting $S$ obtained in step I.

Step II consists of finding this refinement. This is implicit already in the work of Pretzel \cite{Pretzel}, but we give a description here for completeness. It is sufficient to describe step II for a single block of $S$, which we assume to be all of $A$ for ease of notation.
Using linear programming we can determine whether there is a proper subset $I \subseteq [m]$ such that $r(I) = \frac{r([m])}{c([n])}c(N(I))$. In case there are none, the output is just the original block.
In case there is such a block $I$, we output $(I, N(I))$ and $([m]-I, [n]-N(I))$ and apply step II recursively to these two blocks.

To show that this indeed generates $\mcl{D}$, it is sufficient to prove the following lemma, where we have assumed $r([m]) = c([n])$ for ease of notation.
\begin{lemma}
	If the only $I$ satisfying $r(I) \geq c(N(I))$ are $I = \emptyset$ and $[m]$, then $S(A) = S(B)$.	
\end{lemma}
\begin{proof}
	Suppose $(p,q) \in S(B) \backslash S(A)$. 
	By theorem 1 in \cite{Pretzel}, the support of $B$ is the largest possible among all matrices $B\p$ with row sums $r_i$ column sums $c_j$ and satisfying $S(B\p) \subseteq S(A)$.

	Therefore, for any $\e > 0$, there is no matrix with support a subset of $S(A)$, row sums $r\p_i$ and column sums $c\p_j$, where $r\p_i = r_i$ for $i\neq p$, $r\p_i = r_i -\e$ and $c\p_j = c_j$ for $j\neq q$ and $c\p_q = c_q - \e$.
	By lemma \ref{hall}, this means there is some proper subset $I\p = I\p(\e)$ of rows such that $r\p(I\p) > c\p(N(I\p))$.
	Letting $\e\to 0$ (and observing that the number of possible subsets $I\p$ is finite) this gives us a proper subset $I\p$ with $r(I\p) = c(N(I\p))$, a contradiction.
\end{proof}

Now, again referring to \cite{Pretzel}, if change $A$ by setting the entries in $S(A)\backslash S(B)$ to $0$, the ISP limit points will not change. However numerical experiments suggest that convergence is much quicker than without the change.

As a small example of this, take the matrix from the example above, set the entries outside the splitting found ($(1,3)$, $(1,4)$, $(2,1)$, $(2,3)$, and $(2,4)$) to zero. Then it takes about $3$ iterations to get as close to $(B,C)$ in the example, as $B(500)$ and $C(500)$ from earlier are from $(B,C)$.

\section*{Future work}

I would like to mention some related extensions and problems.

First, most results above seem to carry over to the much more general setting of Theorem 5.2 in \cite{CsiShi}, but I have not explored this further.

Second, there is a natural version of the ISP with continuous time, as follows.
Let $\vtheta\in[0,1]$, and define $\R_\vtheta(A)_{ij} = x_i^\vtheta A_{ij}$, and $\C_\vtheta(A)_{ij} = A_{ij}y_j^\vtheta$.
Now we can define $F_{\alpha,\beta} = R_\alpha C_\beta$, and ask about the properties of $\lim_{\e\to 0}\lim_{n\to\infty} F^{(n)}_{\alpha\e, \beta\e}(A)$.
The matrix $F_{\alpha,\beta}(A)$ will be diagonally equivalent to $A$ and from this it follows (cf. \cite{Pretzel}) that the limit $\lim_{\e\to 0}\lim_{n\to\infty} F^{(n)}_{\alpha\e,\beta\e}(A)$ will be the same as the ordinary ISP limit of $A$ if the latter exists.
It is not clear what happens in the general case when that limit does not exist. 

Finally, the ISP can of course be defined for arbitrary, not necessarily nonnegative, $A,r,c$.
Two issues arise in this case.
One problem is that we may obtain matrices having marginals equal to $0$ during the iteration.
This could probably be avoided by using the continuous version described above; it seems reasonable such a system would repel from matrices having some marginal close to $0$.
Also, it will not be possible to prove the analogous statements about the limit points, since they are not true. For example, applying ISP (with discrete time) to the initial data

$\begin{array}{c|cccc}
 	& 4 & 6 \\ \hline
13 & 1 & 2 \\
-12 & 3 & 4 \\
\end{array}$

gives a sequence with period $4$. This cycle appears to be unstable.

\end{document}